\documentclass{amsart}

\newtheorem{thm}{Theorem}[section]

\theoremstyle{remark}
\newtheorem{remark}[thm]{Remark}
\newtheorem{example}[thm]{Example}
\newtheorem{cor}[thm]{Corollary}
\newtheorem{pro}[thm]{Proposition}
\numberwithin{equation}{section}

\def\k{k}
\def\n{n}
\def\tr{\mathop{\rm tr}}
\def\l{\lambda}

\begin{document}

\title[Eigenvalue inequalities]
{Inequalities for selected eigenvalues \\of the product of matrices}

\author[B.-Y. Xi]{Bo-Yan Xi}
\address[B.-Y. Xi]{College of Mathematics, Inner Mongolia University for Nationalities, Tongliao City, Inner Mongolia Autonomous Region, 028043, China}
\email{baoyintu78@imun.edu.cn}

\author[F. Zhang]{Fuzhen Zhang}
\address[F.  Zhang]{Department of Mathematics, Nova Southeastern University, 33301 College Ave., Fort Lauderdale, FL 33314, USA}
\email{zhang@nova.edu}

\begin{abstract}
The product of a Hermitian matrix and a positive semidefinite matrix has only real eigenvalues. We present  bounds for sums of eigenvalues of such a product. 
\end{abstract}

\subjclass[2010]{Primary 15A42; Secondary 47A75}

\keywords{Eigenvalue, Hermitian matrix,  inequality, positive semidefinite matrix}

\thanks{The first author  was  supported in part by the National Natural Science Foundation  of China Grant
No.~11361038.  The second author was   supported  in part by NSU Research Scholar grant.}

\maketitle

\section{Introduction}

Let $A$ be an $n\times n$ Hermitian matrix with eigenvalues
 $\lambda_1(A)\ge\lambda_2(A)\ge \cdots\ge\lambda_n(A)$. If
 some, say $k$, of the eigenvalues of $A$ are selected, they may be indexed by a
  sequence $1\leq i_1<i_2< \dots <i_k\leq n$.  Hence,
  $\lambda_{i_1}(A)\ge\lambda_{i_2}(A)\ge \cdots\ge\lambda_{i_k}(A)$.

  A classical result of Wielandt \cite{Wielandt-PAMS-1955} states that if $A$ and $B$
  are $n\times n$ Hermitian matrices and
 $1\le i_1<i_2<\cdots<i_k\le n$, then
  \begin{equation}\label{Eq:v10Wei}
 \sum_{t=1}^k\l_{i_t}(A+B) \leq \sum_{t=1}^k \l_{i_t}(A)+\sum_{t=1}^k\l_{t}(B).
\end{equation}

A reversed  inequality follows from (\ref{Eq:v10Wei}) by replacing $A$ and $B$ with $-A$ and $-B$, respectively:
\begin{equation}\label{Eq:v10Wei2}
\sum_{t=1}^k \l_{i_t}(A)+\sum_{t=1}^k\l_{n-t+1}(B)\leq \sum_{t=1}^k\l_{i_t}(A+B).
\end{equation}

There is a great amount of research on the partial sums of selected eigenvalues  (see, e.g., \cite{Lid50, Smi68, TF71} or \cite[Chap.~9]{MOA11} and \cite[Chap.~III]{BhaMA97})
as well as on the characterization of the eigenvalues of the sum of  Hermitian matrices (see, e.g., \cite{ BhatiaMonthlyHorn, AHorn62,  KTHorn99, KYHorn01}).

Inequalities analogous to  (\ref{Eq:v10Wei}) and (\ref{Eq:v10Wei2}) for the product of two  matrices are presented in \cite{Wang-Zhang-LAA-1992}:
If $A$ and $B$ are $n\times n$  positive semidefinite   matrices, then
\begin{equation}\label{Eq:WZ}
\sum\limits_{t=1}^k\lambda_{i_t}(A)\lambda_{n-t+1}(B)
\le \sum\limits_{t=1}^k\lambda_{i_t}(AB)\le\sum\limits_{t=1}^k\lambda_{i_t}(A)\lambda_t(B).
\end{equation}

The eigenvalues of the product of two Hermitian matrices need not be real. For example, for $A=\left ( {0 \atop 1}{1 \atop 0} \right )$ and
 $B=\left ( {1 \atop 1}{1 \atop -1} \right ),$  the eigenvalues of $AB$ are
 $1\pm i$. Thus,  inequalities (\ref{Eq:WZ}) do not extend to partial sums of eigenvalues of the product of two Hermitian matrices. However,   requiring one matrix to be positive semidefinite (PSD) ensures that  the eigenvalues of the product are all real; that is,  if $A$ is Hermitian and $B$ is PSD, then
 $AB$ and $B^{1/2}AB^{1/2}$ have the same eigenvalues, so $AB$ has only real eigenvalues.

Eigenvalue problem  is of central importance in matrix analysis and related areas. Usually,  inequalities for selected eigenvalues  involve  two Hermitian matrices for sum and two positive semidefinite matrices  for product.
Nevertheless,
 the results on the {partial sums of  selected eigenvalues} of the product of one PSD matrix and one Hermitian matrix are fragmentary.  
 A result of this kind is, for example, a  celebrated theorem of Ostrowski (see, e.g., \cite[p.~283]{HJ1.13}) which states that, for Hermitian $A$ and positive definite $B$,
 $\l_t(AB)=\theta_t \l_t(A)$, where $\theta_t\in [\l_n(B), \l_1(B)]$.

 The purpose of this paper is to present inequalities on the partial sums of selected eigenvalues of the product of a PSD matrix and a Hermitian matrix. Our results
  generalize some existing ones such as inequalities (\ref{Eq:WZ}).

\section{Eigenvalue inequalities  for Hermitian and PSD matrices}

In \cite[Theorem 3]{IEEEZZ06}, inequalities concerning $\sum_{t=1}^k\l_t(AB)$ are shown,  where $A$ is Hermitian and $B$ is positive semidefinite. These inequalities are about the sum of the  $k$ largest eigenvalues of $AB$.
In this section, we show inequalities for selected eigenvalues; that is, we present
inequalities concerning  $\sum_{t=1}^k\l_{i_t}(AB)$.

We borrow  Wielandt's  min-max representation (see, e.g.,  \cite[p.\,67]{BhaMA97}) for the eigenvalues of  Hermitian matrices,  which is used in the proof of our main result.

\begin{thm}[Wielandt \cite{Wielandt-PAMS-1955}] \label{thm1-2}
If $A\in \mathbb{C}^{n\times n}$ (the set of $n\times n$ complex matrices) is Hermitian  and
$1\le i_1<i_2<\cdots<i_k\le n$, then
\begin{eqnarray*}
\sum\limits_{t=1}^k\lambda_{i_t}(A)
 & = & \max\limits_{S_1\subset\cdots\subset S_k\subset \mathbb{C}^n\hfill \atop  \dim S_t=i_t}
  \min\limits_{\quad {x}_t\in S_t\hfill\atop ({x}_r,{x}_s)=\delta_{rs}}
  \sum\limits_{t=1}^k {{x}_t^*A{x}_t },
\end{eqnarray*}
where $\delta_{rs}$ is the Kronecker delta and $x^*$ is the conjugate transpose of $x\in \Bbb C^n$.
\end{thm}

Denote the inertia of  an $n\times n$  Hermitian $A$ by
$(\pi_+, \nu_{-}, \delta_0)$, where $\pi_+, \nu_-, \delta_0$ are the numbers of
positive, negative and zero eigenvalues of $A$, respectively (see, e.g., \cite[p.\,255]{ZFZbook11}).
  Let $\n_A$ be the number of nonnegative eigenvalues of   $A$, namely,  $\n_A=\pi_++\delta_0$.  For any Hermitian matrix $A\in \Bbb C^{n\times n}$,  we have
\begin{equation}\label{2.1b}
\n_A+\n_{-A}=\pi_++\nu_-+2\delta_0=n+\delta_0.
\end{equation}

Let  $\k_A$  be the number
of nonnegative eigenvalues in $\{\lambda_{i_1}(A), \lambda_{i_2}(A), \dots, \lambda_{i_k}(A)\}$ for
the given index sequence $1\le i_1<i_2<\cdots<i_k\le n$.
So, if $\k_A=0$, 
then all $\lambda_{i_1}(A), \lambda_{i_2}(A), \dots, \lambda_{i_k}(A)$ are negative;
if $\k_A=k$, then each of $\lambda_{i_1}(A), \lambda_{i_2}(A), $ $\dots, \lambda_{i_k}(A)$ is positive or zero.
It is always true that $\k_A\leq \n_A$ and $\lambda_{\k_A}\ge \lambda_{\n_A}$.
\medskip

Now we are ready to present our main theorem.
In what follows,  our convention is that a summation  in the form
 $\sum_{t=p}^q$ over $t$ vanishes if $p>q$.

\begin{thm}\label{thm21}
Let $A\in  \Bbb C^{n\times n} $  be  Hermitian, let $B \in  \Bbb C^{n\times n}$ be  positive semidefinite, and  let $1\le i_1<i_2<\cdots<i_k\le n$.
 Then
 \begin{equation}\label{Eq:thm21-1a}
\sum\limits_{t=1}^k\lambda_{i_t}(AB)\le\sum\limits_{t=1}^{\k_A }
 \lambda_{i_t}(A)\lambda_t(B)+
   \sum\limits_{t=\k_A+1}^k\lambda_{i_{t}}(A)\lambda_{n-k+t}(B)
\end{equation}
and
\begin{equation}\label{thm21-1b}
\sum\limits_{t=1}^k\lambda_{i_t}(AB)\ge \sum\limits_{t=1}^{\k_A}\lambda_{i_t}(A)\lambda_{n-t+1}(B)+
\sum\limits_{t=\k_A+1}^k\lambda_{i_{t}}(A)\lambda_{k-t+1}(B).
\end{equation}
\end{thm}

\begin{proof}
Note that if $k_A=0$, then the first terms (summations)  on the right-hand sides of  (\ref{Eq:thm21-1a}) and (\ref{thm21-1b}) are absent. If $k_A=k$, then the second terms disappear. For $1\le k_A<k$, the first terms on the right-hand sides  of  (\ref{Eq:thm21-1a}) and (\ref{thm21-1b}) are  nonnegative, while the second terms are nonpositive.
\medskip

We divide the proof of the theorem into five cases.

\medskip
\noindent
Case (1). If $\n_A=n$, then $A$  is PSD and $\k_A=k$.
 (\ref{Eq:thm21-1a}) and (\ref{thm21-1b}) are   (\ref{Eq:WZ}).

\medskip
\noindent
Case (2). If $\n_A=0$, then $\k_A=0$ and $-A$  is positive definite.  We set $j_t$ as follows:
 $$1\leq j_1=n-i_k+1<\cdots< j_t=n-i_{k-t+1}+1<\cdots <j_k=n-i_1+1\leq n.$$ 
By  inequalities (\ref{Eq:WZ}), we have
\begin{equation*}\label{thm21-2}
\sum\limits_{t=1}^k\lambda_{j_{t}}(-A)\lambda_{n-t+1}(B)
\le \sum\limits_{t=1}^k\lambda_{j_{t}}(-AB)\le\sum\limits_{t=1}^k\lambda_{j_{t}}(-A)\lambda_t(B).
\end{equation*}
Because  $\lambda _{j_t}(-A)=-\lambda_{n-j_t+1}(A)$ and  $j_t=n-i_{k-t+1}+1$,  we obtain
\begin{equation}\label{thm21-3}
\sum\limits_{t=1}^k\lambda_{i_{k-t+1}}(A)\lambda_t(B)
\le \sum\limits_{t=1}^k\lambda_{i_{k-t+1}}(AB)\le\sum
\limits_{t=1}^k\lambda_{i_{k-t+1}}(A)\lambda_{n-t+1}(B).
\end{equation}
Note that on the right hand side,  $\lambda_{i_{k-t+1}}(A)$ is paired with (multiplied by)  $\lambda_{n-t+1}(B)$, $t=1, 2, \dots, k$, namely, $\lambda_{i_{r}}(A)$ is paired with $\lambda_{n-k+r}(B)$, $r=1, 2, \dots, k$.
Likewise, on the left, $\lambda_{i_{r}}(A)$ is paired with $\lambda_{k-r+1}(B)$,
$r=1, 2, \dots, k$.
It follows that
\begin{equation}\label{Eq:thm21-3b}
\sum\limits_{r=1}^k\lambda_{i_{r}}(A)\lambda_{k-r+1}(B)
\le \sum\limits_{r=1}^k\lambda_{i_{r}}(AB)\le\sum\limits_{r=1}^k\lambda_{i_{r}}(A)\lambda_{n-k+r}(B).
\end{equation}
Inequalities (\ref{Eq:thm21-1a}) and (\ref{thm21-1b}) follow immediately.
\medskip

If $A$ has no positive eigenvalues, then $-A$ is PSD which can be dealt with as above. We assume below that $A$ has both positive and negative eigenvalues.

\medskip
\noindent
Case (3). Let $1\le \pi_+$, $1\le \nu_-$, and $1\le i_1<i_2<\cdots<i_k\le \n_A$. Then
all $\lambda_{i_1}(A), \lambda_{i_2}(A), $ $\dots, \lambda_{i_k}(A)$ are nonnegative, namely,  $\k_A=k$.
 We show that
\begin{equation}\label{thm21-4a}
\sum\limits_{t=1}^k\lambda_{i_t}(A)\lambda_{n-t+1}(B)
\le \sum\limits_{t=1}^k\lambda_{i_t}(AB)\le\sum\limits_{t=1}^k\lambda_{i_t}(A)\lambda_t(B).
\end{equation}

We may assume that $A$ is a diagonal matrix as the inequalities are invariant under unitary similarity. For the upper bound, using the
 approach of splitting (see \cite[p.~381]{LiMathias99} or \cite[p.~250]{HJ1.13}), we write
$A=A_++A_-$, where $A_+$ is the diagonal matrix with all the positive (if any) eigenvalues of $A$ on the main diagonal (plus some zeros), and $A_-$ is the diagonal matrix with all the negative (if any) eigenvalues of $A$ on the main diagonal (plus some zeros).
Then $A_+$ is positive semidefinite and $\lambda_{i_1}(A), \lambda_{i_2}(A),$ $\dots, \lambda_{i_k}(A)$ are all contained on the main diagonal of $A_+$ as $i_k\leq n_A$.

Since
$A\leq A_+$ (i.e., $A_+-A$ is PSD), we have $\l_t(AB)\le \l_t(A_+B)$ for every $t$, and
moreover, $\lambda_{i_t}(A_+)=\lambda_{i_t}(A)$ for $t\leq k$.
For the upper bound, by (\ref{Eq:WZ}), we get
\begin{equation*}
 \sum\limits_{t=1}^k\lambda_{i_t}(AB)\le
 \sum\limits_{t=1}^k\lambda_{i_t}(A_+B)\le\sum\limits_{t=1}^k\lambda_{i_t}(A_+)\lambda_t(B)=
 \sum\limits_{t=1}^k\lambda_{i_t}(A)\lambda_t(B).
\end{equation*}
For the lower bound, we need to use the Wielandt's min-max representation.

Bear in mind that  $\lambda_t(A)\ge 0$, $t=1, 2, \dots,\n_A$, and
$\lambda_{i_t}(A)\ge 0$, $t=1, 2, \dots, k$.
Let $A={\rm{diag}}(\lambda_1 (A), \lambda_2(A), \dots,\lambda_n (A))$.
The standard column unit vectors
$e_i=(0, \dots, 0, 1, 0\cdots, 0)^T$ 
are eigenvectors corresponding to $\l_i(A)$, $i=1, 2, \dots, n$.

We assume $B$ is nonsingular (or use continuity with $B_{\varepsilon}=B+\varepsilon I$, $\varepsilon>0$). Let
\begin{equation*}
S_t=
{\rm{Span}}\big(B^{-\frac{1}{2}} {e}_{1}, B^{-\frac{1}{2}} {e}_{2},
  \dots,B^{-\frac{1}{2}}{e}_{i_t} \big), \quad t=1, 2, \dots, k.
\end{equation*}
Then $S_1\subset S_2\subset  \cdots \subset S_k$ and $\dim (S_t)=i_t$.
The min-max representation reveals  
\begin{equation*}
\sum\limits_{t=1}^k\lambda_{i_t}(AB)
\ge \min\limits_{\substack{ {u}_t  \in S_t  \\ (u_r, u_s) = \delta_{rs} }}
 \sum\limits_{t = 1}^k {u}_t^* B^{\frac{1}{2}} AB^{\frac{1}{2}}{u}_t
= \min\limits_{\substack{ {u}_t  \in S_t  \\ (u_r, u_s) = \delta_{rs} }}
\sum\limits_{t=1}^k\frac{{u}_t^*B^{\frac{1}{2}}AB^{\frac{1}{2}}{u}_t}
{(B^{\frac{1}{2}}u_t)^* (B^{\frac{1}{2}}{u}_t)}\cdot {u}_t^* B{u}_t.
\end{equation*}
For    ${u}_t\in S_t$,  let $u_t=\sum\limits_{j=1}^{i_t} a_jB^{-\frac{1}{2}}{e}_j$,   $a_1, a_2, \dots, a_{i_t}\in\mathbb{C}$, $t=1, 2, \dots, k$.   Then we have
\[
\frac{{u}_t^*B^{\frac{1}{2}}AB^{\frac{1}{2}}{u}_t}
{(B^{\frac{1}{2}}u_t)^* (B^{\frac{1}{2}} {u}_t)}
=\frac{1}{\sum\limits_{j=1}^{i_t} |a_j|^2}\sum\limits_{j=1}^{i_t}|a_j|^2\l_j\ge\lambda_{i_t}(A), \;\; t=1, 2, \dots,k.
\]
Let $C=\sum_{t=1}^k \lambda_{i_t}(A) u_tu_t^*\in \Bbb C^{n\times n}$, where
$\{u_1, u_2, \dots, u_k\}$ is an orthonormal set in $\Bbb C^n$.
Then $C$ is PSD and $\l_t(C)=\l_{i_t}(A)$ for each $t$.
We obtain
\begin{eqnarray*}
{\sum\limits_{t=1}^k\lambda_{i_t}(AB)}
& \ge & \min\limits_{\substack{ {u}_t  \in S_t  \\ (u_r, u_s) = \delta_{rs} }}
\sum\limits_{t=1}^k\lambda_{i_t}(A){u}_t^* B{u}_t=\min\limits_{\substack{ {u}_t  \in S_t  \\ (u_r, u_s) = \delta_{rs} }}\tr (CB)\\
& \ge & \min\limits_{\substack{ {u}_t  \in S_t  \\ (u_r, u_s) = \delta_{rs} }} \sum\limits_{t=1}^k\lambda_{t}(C)\lambda_{n-t+1} (B)
=\sum\limits_{t=1}^k\lambda_{i_t}(A)\lambda_{n-t+1} (B).
\end{eqnarray*}

We note here that it would be nice if we could derive the lower bound from the upper bound with $A$ replaced by $-A$, without using the Wielandt's min-max representation. However, this approach doesn't work as $\l_{i_t}(-A)\leq 0$. Moreover, like the lower bound, the upper bound can also be obtained by using the min-max representation with suitably  selected subspaces.

\medskip
\noindent
Case (4).  Let $1\le \pi_+$, $1\le \nu_-$, and  $\n_A< i_1<i_2<\cdots<i_k\le n$.  Then
none of $\lambda_{i_1}(A), \lambda_{i_2}(A), \dots, \lambda_{i_k}(A)$ is positive.  Consider $-A$ with
$1\le j_1=n-i_k+1<\cdots<j_k=n-i_1+1\leq n$.
Since $j_k=n-i_1+1\leq  n-\n_A\le n_{-A}$ (see (\ref{2.1b})), we apply case (3) to $-A$ and $B$ to get the desired inequalities (\ref{Eq:thm21-3b})  as in case (2).

\medskip
\noindent
Case (5).  Let  $1\le \pi_+$, $1\le \nu_-$, and $1 \le i_1<\cdots<i_{\k_A}\le \n_A<i_{\k_A+1}<\cdots< i_k\le n$. \nolinebreak
Then
$$\sum\limits_{t=1}^k\lambda_{i_t}(AB)=
\sum\limits_{t=1}^{k_A}\lambda_{i_t}(AB)+
\sum\limits_{t=k_A+1}^k\lambda_{i_t}(AB). $$

For $1 \le i_1<\cdots<i_{\k_A}\le \n_A$, using (\ref{thm21-4a}), we have
\begin{equation}\label{thm21-6}
\sum\limits_{t=1}^{\k_A}\lambda_{i_t}(A)\lambda_{n-t+1}(B)
\le\sum\limits_{t=1}^{\k_A}\lambda_{i_t}(AB)\le\sum\limits_{t=1}^{\k_A}\lambda_{i_t}(A)\lambda_t(B).
\end{equation}

For $\n_A<i_{\k_A+1}<\cdots< i_k\le n$,  set $j_t=i_{k_A+t}$, $t=1, 2, \dots, k'$, where $k'=k-k_A$. Then $\n_A<j_{1}<\cdots< j_{k'}\le n$. By case (4),  we have
(for the upper bound)
\begin{eqnarray}\label{thm21-7a}
 \sum\limits_{t=k_A+1}^{k}\lambda_{i_t}(AB) & = & \sum\limits_{t=1}^{k'}\lambda_{j_t}(AB)\notag \\
  &  \le & \sum\limits_{t=1}^{k'}\lambda_{j_{t}}(A)\lambda_{n-k'+t}(B) 
  \notag  \\
   & = & \sum\limits_{t=1}^{k'}\lambda_{i_{k_A+t}}(A)\lambda_{n-k'+t}(B)\notag  \\
    & = & \sum\limits_{r=k_A+1}^{k}\lambda_{i_{r}}(A)\lambda_{n-k+r}(B)
\end{eqnarray}
and (for the lower bound)
\begin{eqnarray}\label{thm21-7b}
\sum\limits_{t=k_A+1}^{k}\lambda_{i_t}(AB) & = & \sum\limits_{t=1}^{k'}\lambda_{j_t}(AB)\notag \\
  &  \ge & \sum\limits_{t=1}^{k'}\lambda_{j_{t}}(A)\lambda_{k'-t+1}(B)\notag  \\
   & = & \sum\limits_{t=1}^{k'}\lambda_{i_{k_A+t}}(A)\lambda_{k'-t+1}(B)\notag  \\
    & = & \sum\limits_{r=k_A+1}^{k}\lambda_{i_{r}}(A)\lambda_{k-r+1}(B).
\end{eqnarray}

Combing  inequalities \eqref{thm21-6}, \eqref{thm21-7a} and \eqref{thm21-7b} results in \eqref{Eq:thm21-1a} and (\ref{thm21-1b}).
\end{proof}

\begin{remark}  Li and Mathias \cite[Theorem 2.3]{LiMathias99} showed some inequalities in partial products of the form
 $\prod_{t=1}^k\frac{\l_{i_t}(AB)}{\l_{i_t}(A)}$, where $A$ is Hermitian and $B=$ 
 is PSD.

It is known that
for positive numbers,
 inequalities of partial products   imply those of partial sums.
 Simply put in the language of majorization, log-majorization implies weak-majorization (see, e.g.,  \cite[p.~232]{HiaiBK14} or \cite[p.~345]{ZFZbook11}). Thus,
 it is tempting to obtain the results on partial sums from Li and Mathias' result on  partial products. However, this is only possible for
  $|{\l_{i_t}(AB)}|$ and $|{\l_{i_t}(A)}|$ by
  observing  that $\frac{\l_{i_t}(AB)}{\l_{i_t}(A)}> 0$. Since
 ${\l_{i_t}(AB)}$ and ${\l_{i_t}(A)}$ are paired with the same sign (maybe both negative), we don't see how the absolute values are dropped so that
 ${\l_{i_t}(AB)}$ and ${\l_{i_t}(A)}$  appear in a partial sum without being paired as a quotient.
 \end{remark}

\begin{remark}
 Hoffman's min-max representation (see, e.g., \cite[2.17]{Amir1990}) in the product form $\prod_{t=1}^k\l_{i_t}(\cdot)$ for positive semidefinite matrices does not generalize to Hermitian matrices
 as Li and Mathias showed by example \cite[pp.~411-412]{LiMathias99}; that is,  there is no multiplicative analog  of Wielandt's min-max representation (Theorem \ref{thm1-2}) 
 for Hermitian matrices. Therefore, it is impossible to
 derive sum inequalities  $\sum_{t=1}^k\l_{i_t}(\cdot)$ from product inequalities $\prod_{t=1}^k\l_{i_t}(\cdot)$ for Hermitian matrices through majorization.
 In view of this,  our Theorem \ref{thm21} appears to be  important.
\end{remark}

\begin{example} Let $n=3$,  $k=2$, and let
\begin{align*}
A = \left({\begin{array}{*{20}c}
   1 & 2 & 0 \\
   2 & 1 & 0 \\
   0 & 0 & -4
\end{array}}\right), \quad B =\left({\begin{array}{*{20}c}
   2   & -1& 0 \\
   - 1 & 2 & 0 \\
   0   & 0 & 2
\end{array}} \right).
\end{align*}
Then
$$\begin{array}{lll}
\lambda_{1}(A)=3, & \lambda_{2}(A)=-1, & \lambda_{3}(A)=-4;  \\
\lambda_{1}(B)=3,& \lambda_{2}(B)=2, & \lambda_{3}(B)=1;  \\
\lambda_{1}(AB)=3,&\lambda_{2}(AB)=-3, & \lambda_{3}(AB)=-8.
\end{array}$$
Thus, $n_A=1$.
We consider the cases (I) $i_1=1,$ $i_2=2,$ $ k_A=1$; (II) $i_1=1,$ $ i_2=3,  $ $k_A=1$; and (III) $i_1=2,$ $ i_2=3,$ $ k_A=0$, to get, respectively,
\begin{eqnarray*}
\lefteqn{\sum\limits_{t=1}^{k_A } \lambda_{i_t}(A)\lambda_{t}(B)
   +\sum\limits_{t=k_A+1}^k\lambda_{i_{t}}(A)\lambda_{n-k+t}(B) }\\
  & & =\begin{cases}
   \lambda_{1}(A)\lambda_{1}(B)+\lambda_{2}(A)\lambda_{3}(B)=8, & \mbox{(I)} \\
   \lambda_{1}(A)\lambda_{1}(B)+\lambda_{3}(A)\lambda_{3}(B)=5,
       & \mbox{(II)} \\
    \lambda_{2}(A)\lambda_{2}(B)+\lambda_{3}(A)\lambda_{3}(B)=-6,
      & \mbox{(III)}\end{cases}
\end{eqnarray*}

$$
\sum\limits_{t=1}^k\lambda_{i_t}(AB)=\begin{cases}
   \lambda_{1}(AB)+\lambda_{2}(AB)
=0, &\mbox{(I)}\\
   \lambda_{1}(AB)+\lambda_{3}(AB)
=-5,&\mbox{(II)} \\
    \lambda_{2}(AB)+\lambda_{3}(AB)
=-11,&\mbox{(III)}
\end{cases}
$$
and
\begin{eqnarray*}
\lefteqn{\sum\limits_{t=1}^{k_A}\lambda_{i_t}(A)\lambda_{n-t+1}(B)+\sum\limits_{t=k_A+1}^k\lambda_{i_{t}}(A)\lambda_{k-t+1}(B)}\\
 & & =\begin{cases}
   \lambda_{1}(A)\lambda_{3}(B)+\lambda_{2}(A)\lambda_{1}(B)=0, &\mbox{(I)}  \\
   \lambda_{1}(A)\lambda_{3}(B)+\lambda_{3}(A)\lambda_{1}(B)=-9,
    &\mbox{(II)}\\
    \lambda_{2}(A)\lambda_{2}(B)+\lambda_{3}(A)\lambda_{1}(B)=-14.
       &\mbox{(III)}
\end{cases}
\end{eqnarray*}
\end{example}

\medskip

Recall the
 trace inequality (see, e.g., \cite[p.\,78]{BhaMA97} or \cite[p.\,255]{HJ1.13})
 for Hermitian $A, B$:
\begin{equation}\label{HJp255}
\sum_{t=1}^n \l_t(A)\l_{n-t+1}(B)\leq \tr (AB)=\sum_{t=1}^n\l_t(AB) \leq \sum_{t=1}^n \l_t(A)\l_{t}(B).
\end{equation}

As we noted in Section 1, the product of two Hermitian matrices may have nonreal eigenvalues. So, the trace in (\ref{HJp255}) cannot be replaced by partial sums in general.  
Theorem \ref{thm21} presents partial sum inequalities for the product of
one Hermitian matrix and one PSD matrix.

\medskip

We present a few  results that are immediate from Theorem \ref{thm21}.

\begin{cor}\label{cor21-1-2}
Let $A \in \mathbb{C}^{n\times n}$ be a stable Hermitian  matrix {\rm (}i.e., the eigenvalues of $A$ are located on the left  half-plane{\rm)} and
  $B\in \mathbb{C}^{n\times n}$  be  positive semidefinite. Then
\begin{equation*}\label{cor21-1}
\sum\limits_{t=1}^k\lambda_{i_{t}}(A)\lambda_{k-t+1}(B)
\le \sum\limits_{t=1}^k\lambda_{i_t}(AB)\le\sum\limits_{t=1}^k\lambda_{i_{t}}(A)\lambda_{n-k+t}(B).
\end{equation*}
\end{cor}

\begin{proof} Note that $k_A=0$ or $\lambda_{i_t}(A)=0$ for $t=1, 2, \dots, \k_A$ in \eqref{Eq:thm21-1a} and (\ref{thm21-1b}).
\end{proof}

\begin{cor}\label{cor21-1-3Z}
Let $A\in \mathbb{C}^{n\times n}$ be Hermitian  and $B\in \Bbb C^{n\times n}$ be   positive semidefinite. Then for any positive eigenvalue $\lambda_s(AB)$ and  negative eigenvalue $\lambda_t(AB)$ (if any)
\begin{equation*}\label{thm21-10Z}
\lambda_{s}(A)\lambda_{n}(B)+\lambda_{t}(A)\lambda_{1}(B)
\le \lambda_{s}(AB)+\lambda_{t}(AB)\le\lambda_{s}(A)\lambda_1(B)+
\lambda_{t}(A)\lambda_n(B).
\end{equation*}
\end{cor}

\begin{proof}
Take $k=2$ and $1\leq i_1=s<i_2=t\leq n$. Then $\lambda_s(A)>0>\lambda_{t}(A)$.
In \eqref{Eq:thm21-1a}, there is one term in each of the two summations; the same is true for (\ref{thm21-1b}).
\end{proof}

\begin{cor}\label{cor21-1-3}
Let $A\in \mathbb{C}^{n\times n}$ be Hermitian  and $B\in \Bbb C^{n\times n}$ be   positive semidefinite. If $\lambda_{p}(AB)$ is the smallest positive eigenvalue of $AB$ and $\lambda_{q}(AB)$ is the largest negative eigenvalue of $AB$ (if any), then their distance (gap near 0)  is bounded as
$$\lambda_{p}(AB)-\lambda_{q}(AB)\leq \left [ \lambda_{p}(A)-\lambda_{q}(A)\right ]\lambda_1(B). $$
\end{cor}

\begin{proof} Since $\lambda_{p}(AB)>0$  and $\lambda_{q}(AB)<0$, we have $\lambda_{p}(A)>0>\lambda_{q}(A)$.
Setting  $k=1$ and $i_1=t$, for each $t=1, 2,  \dots, n$ in Theorem \ref{thm21}, 
 we obtain
\begin{equation*}\label{thm21-10a}
\lambda_{t}(A)\lambda_{n}(B)
\le \lambda_{t}(AB)\le\lambda_{t}(A)\lambda_1(B), \quad \mbox{if $\lambda_{t}(A)\ge0$}
\end{equation*}
and
\begin{equation*}\label{thm21-11}
\lambda_{t}(A)\lambda_{1}(B)
\le \lambda_{t}(AB)\le\lambda_{t}(A)\lambda_n(B), \quad \mbox{if $\lambda_{t}(A)\le0$}.
\end{equation*}
In particular,
$$\lambda_{p}(AB)\le\lambda_{p}(A)\lambda_1(B)\quad \mbox{and}
\quad \lambda_{q}(A)\lambda_{1}(B) \le \lambda_{q}(AB).$$
The desired inequality follows immediately by subtraction.
\end{proof}

Corollary \ref{cor21-1-3} provides  an estimate for the gap between two  eigenvalues of $AB$ near zero (on both sides) in terms of the eigenvalues of  $A$ and $B$.  The ratio $[\lambda_{p}(AB)-\lambda_{q}(AB)]/ [\lambda_{p}(A)-\lambda_{q}(A)]$ is bounded above by  the spectral norm of $B$. Thus,  if neither $A$ nor $-A$ is  stable Hermitian,
then an application (multiplication) of a strictly contractive PSD matrix $B$ to $A$ narrows
the gap
between the positive and negative eigenvalues of $A$.
 (A matrix is said to be strictly contractive if its spectral norm is less  than one.)


\section{Comparison of the bounds}
Recall from the proof of Theorem \ref{thm21} the splitting of the real diagonal matrix $A$:  
$A=A_++A_-$, where $A_+$ is the diagonal matrix with all  positive (if any) eigenvalues of $A$ on the main diagonal (plus some zeros), and $A_-$ is the diagonal matrix with all negative (if any) eigenvalues of $A$ on the main diagonal (plus some zeros). It is natural to ask what inequalities would be derived if one applies
(\ref{Eq:v10Wei}), (\ref{Eq:v10Wei2}),  and (\ref{Eq:WZ}) to $AB=A_+B+A_-B$. We present the inequalities obtained by  the splitting approach in the following proposition; we then show that  these inequalities   are
in general
weaker than the ones in Theorem~\ref{thm21}. We show the case for upper bound.

\begin{pro}\label{thm21New}
Let $A\in \mathbb{C}^{n\times n}$ be Hermitian  and let $B\in \Bbb C^{n\times n}$ be   positive semidefinite. For $1\le i_1<i_2<\cdots<i_k\le n$, we have
 \begin{equation}\label{Eq:thm21-1aZ}
\sum\limits_{t=1}^k\lambda_{i_t}(AB)\le\sum\limits_{t=1}^{\k_A }
 \lambda_{i_t}(A)\lambda_t(B)+
   \sum\limits_{t=n_A+1}^k\lambda_{t}(A)\lambda_{n-k+t}(B).
\end{equation}
\end{pro}

\proof
Note that $A_+$ is positive semidefinite and all $\lambda_{i_1}(A), \lambda_{i_2}(A), $ $\dots, \lambda_{i_k}(A)$ are contained on the main diagonal of $A_+$. So, $\l_{i_t}(A_+)=\l_{i_t}(A)$ for $t=1, 2, \dots, k_A$, and $\l_{i_t}(A_+)=0$ for $t>k_A$.
Observe $\l_t(A_-)=0$ for $t\leq n_A$.
It follows that
 \begin{eqnarray*}
{}\hspace{.65in}
\sum\limits_{t=1}^k\lambda_{i_t}(AB) & = & \sum\limits_{t=1}^k\lambda_{i_t}(A_{+}B+A_{-}B) \\
   &  \leq  &  \sum\limits_{t=1}^k\lambda_{i_t}(A_{+}B)+\sum\limits_{t=1}^k\lambda_{t}(A_{-}B)\\
      & \leq  & \sum\limits_{t=1}^k\lambda_{i_t}(A_{+})\l_t(B)+\sum\limits_{t=1}^k\lambda_{t}(A_{-})\l_{n-k+t}(B)\\
       &= & \sum\limits_{t=1}^{k_A}\lambda_{i_t}(A)\l_t(B)
       +\sum\limits_{t=n_A+1}^k\lambda_{t}(A)\l_{n-k+t}(B).  {}\hfill \qquad \qed
\end{eqnarray*}

In comparison, the first term on the right hand side of (\ref{Eq:thm21-1aZ}) is the same as that of (\ref{Eq:thm21-1a}), while the second terms are different. We denote the second term in (\ref{Eq:thm21-1a})
by $T_1$ and the second term in (\ref{Eq:thm21-1aZ})
by  $T_2$. We show that $T_1\leq T_2$ as follows.
\begin{eqnarray*}
T_1 & = &  \sum\limits_{t=k_A+1}^k\lambda_{i_t}(A)\lambda_{n-k+t}(B) \\
& = &  \sum\limits_{t=k_A+1}^{n_A}\lambda_{i_t}(A)\lambda_{n-k+t}(B)+ \sum\limits_{t=n_A+1}^k\lambda_{i_t}(A)\lambda_{n-k+t}(B)\\
& \leq  &    \sum\limits_{t=n_A+1}^k\lambda_{i_t}(A)\lambda_{n-k+t}(B)\\
& \leq  &  \sum\limits_{t=n_A+1}^k\lambda_{t}(A)\lambda_{n-k+t}(B)=T_2.
\end{eqnarray*}
\medskip

\noindent
{\bf Acknowledgement.} The second author appreciates discussions with Roger Horn, Chi-Kwong Li, and Ren-cang Li during the preparation of the manuscript.

\end{document}